\documentclass{amsart}

\usepackage{graphicx}
\usepackage{amsmath}
\usepackage{verbatim}

\newtheorem{theorem}{Theorem}[section]
\newtheorem{lemma}[theorem]{Lemma}

\theoremstyle{definition}

\theoremstyle{remark}

\numberwithin{equation}{section}

\begin{document}

\title{Configurations of  linear spaces of codimension two and the containment problem}
\author{Ben Drabkin}

\maketitle
\section{Introduction}
Let $I\subset k[x_0,\dots,x_N]$ be a homogeneous ideal.  For $r\geq 0$, the $r$-th symbolic power of $I$ is defined to be \[I^{(r)}=\bigcap_{p\in\mbox{Ass}(R/I)}(I^rR_p\cap R).\]  Symbolic powers of ideals are interesting for a number of reasons, not least of which is that, for a radical ideal $I$,the $r$-th symbolic power $I^{(r)}$ is the ideal of all polynomials vanishing to order at least $r$ on $V(I)$ (by the Zariski-Nagata theorem).

Containment relationships between symbolic and ordinary powers are a source of great interest. As an immediate consequence of the definition, $I^r\subseteq I^{(r)}$ for all $r$.  However, the other type of containment, namely that of a symbolic power in an ordinary power is much harder to pin down.  It has been proved by Ein-Lazarsfeld-Smith \cite{ELS} and Hochster-Huneke \cite{HH} that $I^{(m)}\subseteq I^r$ for all $m\geq Nr$, but as of yet there are no examples in which this bound is sharp.\\

It was conjectured by Harbourne in \cite[Conjecture 8..4.3]{PSC} (and later in  \cite[Conjecture 4.1.1]{HaHu} in the case $e=N-1$) that $I^{(m)}\subseteq I^r$ for all $m\geq er-(e-1)$, where $e$ is the codimension of $V(I)$.  While this conjecture holds in a number of important cases, some counterexamples have also been found.  Notably, the main counterexamples come from singular points of hyperplane arrangements \cite{BNAL}. One particular family is known in the  literature under the name of Fermat configurations of points cf. $\mathbb P^2$ \cite{BNAL, MS}. These have been recently generalized to Fermat-like configurations of lines in $\mathbb P^3$ in \cite{MS}. The $\rm{Ceva}(n)$ arrangement of hyperplanes in $\mathbb{P}^N$ is defined by the linear factors of \[F_{N,n}=\prod_{0\leq i<j\leq N}(x_i^n-x_j^n),\] where $n\geq 3$ is an integer.

In \cite{DST} Dumnicki, Szemberg, and Tutaj-Gasinska showed that, for
the ideal $I_{2,3}$ corresponding to all triple intersection points of
the lines defined by linear factors of $F_{2,3}$ in $\mathbb{P}^2$,
$F_{2,3}\not\in I_{2,3}^2$, but $F_{2,3}\in I_{2,3}^{(3)}$.  This was the first
counterexample to the above mentioned conjecture. Later, in \cite{MS} Malara and
Szpond generalized this construction to $\mathbb{P}^3$, by showing
that for the ideal $I_{3,n}$, corresponding to all triple intersection
lines of the planes defined by the linear factors of $F_{3,n}$,
$F_{3,n}\not\in I_{3,n}^2$, but $F_{3,n}\in I_{3,n}^{(3)}$.  In the following, the
construction of counterexamples to $I^{(3)}\subseteq I_2$ from
Fermat arrangements is generalized to $\mathbb{P}^N$ for all $N\geq
2$.

\section{Main result}
Let $n\in\mathbb{N}$.  Let $k$ be a field 
which contains a primitive $n$-th root of unity, $\varepsilon$.  For each $N\in\mathbb{N}$, let $S_N:=k[x_0,x_1,\dots,x_N]$, and define \[F_{N,n}:=\prod_{0\leq i<j\leq N}(x_i^n-x_j^n).\]
Let \[C_{N,n}:=\bigcap_{0\leq i<j\leq N}(x_i,x_j)\]\[J_{N,n}:=\bigcap_{\substack{0\leq i<j<l\leq N \\ 0\leq a,b< n}}(x_i-\varepsilon^a x_j, x_i-\varepsilon^b x_l),\]
and let
\[I_{N,n}:=J_{N,n}\cap C_{N,n}.\]
We show in Lemma \ref{lem1} that $I_{N,n}$ is the ideal of the $N-2$ dimensional flats arising from triple intersection of hyperplanes corresponding to linear factors of $F_{N,n}$.

\begin{theorem}\label{main}
For all $N\geq 2$, $I_{N,n}^{(3)}\not\subseteq I_{N,n}^2$.
\end{theorem}
Before we can prove this, we must introduce a few lemmas.

\begin{lemma}\label{lem1}
The ideal $I_{N,n}$ defined above defines the union of all the $N-2$ dimensional linear spaces that are intersections of at least three hyperplanes corresponding to linear factors of $F_{N,n}$.
\end{lemma}

\begin{proof}
Let $0\leq a,b<n$, and let $0\leq i<j\leq l\leq N$.  Then \[(x_i-\varepsilon^ax_j,x_i-\varepsilon^bx_l)=(x_i-\varepsilon^ax_j,x_i-\varepsilon^bx_l,x_j-\varepsilon^{b-a}x_l)\] defines the intersection of the three hyperplanes corresponding to $(x_i-\varepsilon^ax_j)$, $(x_i-\varepsilon^bx_l)$, and $(x_j-\varepsilon^{b-a}x_l)$.  
Furthermore $$(x_i,x_j)=(x_i-\varepsilon^ax_j : a=0,1,\dots, n),$$ so $(x_i,x_j)$ defines the intersection of  $n$ hyperplanes corresponding to linear factors of $x_i^n-x_j^n$.  

It remains to be seen that all $N-2$ dimensional linear spaces that arise as intersections of at least three hyperplanes corresponding to linear factors of $F_{N,n}$ are accounted for above. Let $L$ be the ideal defining such a linear space. then $L$ contains three linearly dependent binomials of the form $x_i-\varepsilon^ax_j, x_k-\varepsilon^bx_l, x_u-\varepsilon^cx_v$. Without loss of generality (after multiplication by appropriate powers of $\varepsilon$) this yields $i=k$ and $\{j,l\}=\{u,v\}$. If $j\neq l$ then $L=(x_i-\varepsilon^ax_j, x_i-\varepsilon^bx_l)$ is one of the primes appearing in the decomposition of $J_{N,n}$ and if $j=l$ then $L=(x_i,x_j)$ is one of the primes appearing in the decomposition of $C_{N,n}$.
\end{proof}

\begin{lemma}\label{lem2}
Let $R$ and $S$ be finitely generated graded-local Noetherian rings. Let $m$ be the homogeneous maximal ideal of $R$. Let  $I\subset R$ be a homogeneous ideal, and suppose  $F\not\in I^r$ for some $r\in\mathbb{N}$.  Let $J\subset S$ be an ideal, and let $\pi:S\rightarrow R$ be a (not necessarily homogeneous) ring homomorphism such that $\pi(J)\subseteq I$.
If  $G\in R$ is such that $\pi(G)=Fg$, where $g\not\in m$, then $G\not\in J^r$.
\end{lemma}
\begin{proof}
Suppose by way of contradiction that such a $G$ exists and $G\in J^r$.  Then \[Fg=\pi(G)\in\pi(J^r)\subseteq(\pi(J))^r\subseteq I^r.\]
Thus $Fg\in I^r$.
Let $I^r=Q_1\cap\dots\cap Q_t$ be a primary decomposition.  Then, $Fg\in Q_i$ for each $i\in\{1,\dots,t\}$. Suppose that, for some $i$, $F\not\in Q_i$.  Then $g^s\in Q_i$ for some $s\in\mathbb{N}$.  However, $I^r$ is a homogeneous ideal, so $Q_i\subseteq m$ for all $i\in\{1,2,\dots, t\}$.  But since $g\not\in m$, we know that $g^s\not\in m$, a contradiction.  Thus $F\in Q_i$ for all $i$, which yields $F\in I^r$, a contradiction.
\end{proof}

This lemma  allows us to construct an inductive argument for the main theorem.

\begin{proof}[Proof of Theorem \ref{main}]
By Lemma \ref{lem1} $F_{N,n}$ must vanish to order 3 or greater on each of the linear spaces whose union is $V(I_{N,n})$,  thus $F_{N,n}\in I_{N,n}^{(3)}$. To finish the proof, it suffices to show that for all $N\geq 2$, $F_{N,n}\not\in I_{N,n}^2$.

We argue by induction on $N$.  For $N=2$, this is proved in the paper of Dumnicki, Szemberg, and Tutaj-Gasinska \cite{DST}.

For $N>3$, assume that $F_{N-1,n}\not\in I_{N-1,n}^2$ and consider the evaluation homomorphism $\pi:S_N\rightarrow S_{N-1}$ defined by $\pi(x_N)=1$ and $\pi(x_i)=x_i$ for $i \leq N-1$.  Then:
\[\pi(I_{N,n})\subseteq C_{N-1,n}\cap\left(\bigcap_{0\leq i<N}(x_i,1)\right)\cap J_{N-1,n}\cap \left(\bigcap_{\substack{0\leq i<j<N\\0\leq a,b<n}}(x_i-\varepsilon^a,x_i-\varepsilon^b)\right)\subseteq I_{N-1,n}.\]

We note that $\pi(F_{N,n})=F_{N-1,n}g$ where $g=\prod_{0\leq i<N}(x_i^n-1)$.  Since $g\not\in (x_0,\dots,x_{N-1})$, we conclude by Lemma \ref{lem2} that $F_{N,n}\not\in I_{N,n}^2$.
\end{proof}

\section{Concluding Remarks}
Another proof for the noncontainment noncontainment  $I_{N,n}^{(3)}\subseteq I_{N,n}^2$ has been found by Grzegorz Malara and Justinya Szpond and can bee seen in their upcoming paper \cite{MS2}.

\end{document}